\documentclass[12pt]{article}
\usepackage{amsmath,amsthm,amssymb}

\normalsize

\def\Reals{{\mathbb{R}}}

\newtheorem{theorem}{Theorem}[section]
\newtheorem{lemma}[theorem]{Lemma}
\newtheorem{corollary}[theorem]{Corollary}
\numberwithin{equation}{section}

\def\N{{\mathcal{N}}}
\def\M{{\mathcal{M}}}
\def\S{{\mathcal{S}}}
\def\C{{\mathcal{C}}}
\def\J{{\mathcal{J}}}

\def\P{{\mathcal{P}}}

\def\({\bigl(}   \def\){\bigr)}
\def\abs#1{\mathopen|#1\mathclose|} \let\card=\abs
 \let\Card=\Abs

\def\svec{{\boldsymbol{s}}}
\def\tvec{{\boldsymbol{t}}}
\def\avec{{\boldsymbol{a}}}
\def\ff#1#2{[#1]_{#2}}
\def\ac#1{\card{\C_{#1}}}
\def\dfrac#1#2{\lower0.15ex\hbox{\large$\frac{#1}{#2}$}}

\def\E{{\mathbb{E}}}
\def\Deltait{{\mathit{\Delta}}}
\def\nicebreak{\vskip0pt plus50pt\penalty-300\vskip0pt plus-50pt }

\def\Mst{{\M(\svec,\tvec)}}
\def\Pst{{\P(\svec,\tvec)}}

\title{Asymptotic enumeration of sparse nonnegative\\
integer matrices with specified row and column sums}

\author{
Catherine Greenhill\\
\small School of Mathematics and Statistics\\[-0.5ex]
\small The University of New South Wales\\[-0.5ex]
\small Sydney NSW 2052, Australia\\[-0.5ex]
\small \tt csg@unsw.edu.au\\
\and
Brendan D. McKay\\
\small Research School of Computer Science\\[-0.5ex]
\small Australian National University\\[-0.5ex]
\small Canberra, ACT 0200, Australia\\
\small\tt bdm@cs.anu.edu.au
}

\date{\small Keywords:  asymptotic enumeration, non-negative integer matrices,
              contingency tables, switchings\\
MSC 2000: 05A16, 05C50, 62H17}

\begin{document}

\maketitle

\begin{abstract}
Let $\svec = (s_1,\ldots ,s_m)$ and $\tvec = (t_1,\ldots ,t_n)$
be vectors of nonnegative integer-valued functions of $m,n$ with
equal sum $S = \sum_{i=1}^m s_i = \sum_{j=1}^n t_j$.
Let $M(\svec,\tvec)$ be the number of $m\times n$ matrices
with nonnegative integer entries such that the
$i$th row has row sum $s_i$ and the $j$th column has
column sum $t_j$ for all~$i,j$.
Such matrices occur in many different settings, an important example
being the contingency tables (also called frequency
tables) important in statistics.
Define $s=\max_i s_i$ and $t=\max_j t_j$.
Previous work has established the asymptotic value of
$M(\svec,\tvec)$ as $m,n\to\infty$ with
$s$ and $t$ bounded (various authors independently,
1971--1974), and when
all entries of $\svec$ equal $s$, all entries of $\tvec$
equal $t$, and
$m/n,n/m,s/n\ge c/\log n$ for sufficiently large~$c$ (Canfield
and McKay, 2007).  In this paper we extend the sparse range to
the case $st=o(S^{2/3})$.  The proof in part follows a previous
asymptotic enumeration of 0-1 matrices under the same conditions
(Greenhill, McKay and Wang, 2006).
We also generalise the enumeration
to matrices over any subset of
the nonnegative integers that includes 0~and~1.

\emph{Note added in proof, 2011:}\  This paper appeared in
\emph{Advances in Applied Mathematics} {\bf 41} (2008), 459--481.
Here we fix a small gap in the proof of Lemma 5.1 and
make some other minor corrections.
We emphasise that the statements of our results have not changed.  
\end{abstract}

\section{Introduction}

Let $\svec = \svec(m,n) = (s_1,\ldots ,s_m)$
and $\tvec = \tvec(m,n) = (t_1,\ldots ,t_n)$
be vectors of nonnegative integers with
equal sum $S = \sum_{i=1}^m s_i = \sum_{j=1}^n t_j$.
Let $\Mst$ be the set of all $m\times n$ matrices
with nonnegative integer entries such that the
$i$th row has row sum $s_i$ and the $j$th column has
column sum $t_j$ for each~$i,j$.
Then define $M(\svec,\tvec) = \card{\Mst}$ to be the number of
such matrices.

Our task in this paper is to determine the asymptotic value of
$M(\svec, \tvec)$ as $m,n\to\infty$ under suitable conditions
on $\svec$ and $\tvec$.

The matrices $\Mst$ appear in many combinatorial contexts;
see Stanley~\cite[Chapter~1]{stanley} for a brief history.
A large body of statistical literature is devoted to them
under the name of \textit{contingency tables} or
\textit{frequency tables}; see \cite{DE, DGang} for
a partial survey.
In theoretical computer science there has been interest in
efficient algorithms for the problem of generating
contingency tables with prescribed margins at random,
and for approximately counting these tables.
See for example~\cite{BSY, DKM, morris}.

The history of the enumeration problem 
for nonnegative integer matrices is surveyed in~\cite{CMinteger},
while a history for the corresponding problem
for 0-1 matrices  is given in~\cite{GMW}.
Here we recall only the few previous exact results on
asymptotic enumeration for nonnegative integer
matrices.
Define $s=\max_i s_i$ and $t=\max_j t_j$.
The first non-trivial case $s_1=\cdots=s_m=t_1=\cdots=t_n=3$
was solved by Read~\cite{read} in 1958.
During the period 1971--74, 
this was generalised to bounded~$s,t$ by three independent groups:
B\'ek\'essy, B\'ek\'essy and Koml\'os~\cite{bekessy},
Bender~\cite{bender}, and Everett and Stein~\cite{everett},
under slightly different conditions.

In the case of denser matrices, the only precise asymptotics were
found by Canfield and McKay~\cite{CMinteger} in the case that
the row sums are all the same and the column sums are all the same.
Let $M(m,s;n,t)=M((s,s,\ldots,s),(t,t,\ldots,t))$, where the
vectors have length $m$, $n$, respectively, and $ms=nt$.

\begin{theorem}[{\cite[Theorem 1]{CMinteger}}]\label{CMmain}
Let $s=s(m,n)$, $t=t(m,n)$ be positive integers satisfying $ms=nt$.
Define $\lambda = s/n = t/m$.  Let $a,b>0$ be constants
such that $a+b<\frac12$.  Suppose that $m,n \rightarrow \infty$
in such a way that
\begin{equation} \label{Hyp}
\frac{(1+2\lambda)^2}{4\lambda(1+\lambda)}
 \biggl( 1 + \frac{5m}{6n} + \frac{5n}{6m} \biggr)
\le a \log n.
\end{equation}
Define $\Deltait(m,s;n,t)$ by
\begin{align*}\label{mainformula}
M(m,s;n,t) &=
     \frac{\displaystyle\binom{n+s-1}{s}^{\!m}\binom{m+t-1}{t}^{\!n}}
          {\displaystyle\binom{mn+\lambda mn-1}{\lambda mn}}\\
     &{\kern20mm}{}\times
       \Bigl(\frac{m+1}{m}\Bigr)^{(m-1)/2}
     \Bigl(\frac{n+1}{n}\Bigr)^{(n-1)/2}
     \exp\Bigl(-\dfrac12 + \frac{\Deltait(m,s;n,t)}{m+n}\Bigr).
\end{align*}
Then $\Deltait(m,s;n,t)=O(n^{-b})(m+n)$ as $m,n\to\infty$.\quad\qedsymbol
\end{theorem}
Canfield and McKay conjectured that in fact $0<\Deltait(m,s;n,t)<2$
for all $s,t\ge 1$.  The results in the present paper establish
that conjecture for sufficiently large $m,n$ in the case
$st=o\((mn)^{1/5}\)$. (See Corollary~\ref{Delta}.)

\bigskip

The main result in this paper is the asymptotic value of
$M(\svec,\tvec)$ for $st=o(S^{2/3})$.
Our proof uses the method of switchings in a number of different ways.  
In several aspects our approach is parallel to that
which provided our previous asymptotic
estimate of $N(\svec,\tvec)$, the number of 0-1 matrices in the class
$\Mst$. We now restate that result for convenience.
For any $x$, define $[x]_0 = 1$ and for a positive integer $k$,
$[x]_k = x(x-1)\cdots (x-k+1)$.  Also define
\[ S_k = \sum_{i=1}^m \,[s_i]_k,\qquad
   T_k = \sum_{j=1}^n \,[t_j]_k\]
for $k\geq 1$.  Note that $S_1 = T_1 = S$.

\begin{theorem}[{\cite[Corollary 5.1]{GMW}}]\label{main01}
~Let $\svec = \svec(m,n) = (s_1,\ldots ,s_m)$
and $\tvec = \tvec(m,n) = (t_1,\ldots ,t_n)$
be vectors of nonnegative integers with
equal sum $S = \sum_{i=1}^m s_i = \sum_{j=1}^n t_j$.
Suppose that $m,n\to\infty$, $S\to\infty$
and\/ $1\le st=o(S^{2/3})$.  Then
\begin{align*}
N(\svec,\tvec) 
    & = \frac{S!}{\prod_{i=1}^m s_i!\, \prod_{j=1}^n t_j!}
    \,\exp\biggl(-\frac{S_2T_2}{2S^2} - \frac{S_2T_2}{2S^3} +
    \frac{S_3T_3}{3S^3}
        - \frac{S_2T_2(S_2+T_2)}{4S^4} \\
           &{\kern 60mm} - \frac{S_2^2T_3+S_3T_2^2}{2S^4} + \frac{
             S_2^2T_2^2}{2S^5}
                    + O\biggl(\frac{s^3t^3}{S^2}\biggr)\biggr).
                      \quad\qedsymbol
\end{align*}
\end{theorem}

\bigskip

We now state our main result, which is the asymptotic value
of $M(\svec,\tvec)$ for sufficiently sparse matrices.  Note
that the answer is obtained by multiplying the expression
for $N(\svec,\tvec)$ from Theorem~\ref{main01} by a simple
adjustment factor.

\begin{theorem}\label{main}
~Let $\svec = \svec(m,n) = (s_1,\ldots ,s_m)$
and $\tvec = \tvec(m,n) = (t_1,\ldots ,t_n)$
be vectors of nonnegative integers with
equal sum $S = \sum_{i=1}^m s_i = \sum_{j=1}^n t_j$.
Suppose that $m,n\to\infty$, $S\to\infty$
and\/ $1\le st=o(S^{2/3})$.  Then
\begin{align*}
M(\svec,\tvec) &= N(\svec,\tvec)
  \exp\biggl(\frac{S_2T_2}{S^2} + \frac{S_2T_2}{S^3}
    + O\biggl(\frac{s^3t^3}{S^2}\biggr)\biggr) \\[1ex]
    & = \frac{S!}{\prod_{i=1}^m s_i!\, \prod_{j=1}^n t_j!}
    \,\exp\biggl(\frac{S_2T_2}{2S^2} + \frac{S_2T_2}{2S^3} +
    \frac{S_3T_3}{3S^3}
        - \frac{S_2T_2(S_2+T_2)}{4S^4} \\
           &{\kern 64mm} - \frac{S_2^2T_3+S_3T_2^2}{2S^4} + \frac{
             S_2^2T_2^2}{2S^5}
                    + O\biggl(\frac{s^3t^3}{S^2}\biggr)\biggr).
                      \end{align*}
\end{theorem}

\begin{proof}
The proof of this theorem is presented in Sections~\ref{s:matrices}
and~\ref{s:pairings}.  
First we show that the set of matrices in $\Mst$ 
with an entry greater than 3 forms a vanishingly small proportion
of $\Mst$.   We also show that it is very unusual
for an element of $\Mst$ to have a ``large''
number of entries equal to~2 or a ``large'' number of entries equal 
to~3, where ``largeness'' is defined in Section~\ref{s:matrices}.
We establish these facts using switchings
on the matrix entries.  This allows us to 
concentrate on matrices in $\Mst$ with no entries 
greater than 3 and not very many entries equal to 2 or 3.

We then proceed in Section~\ref{s:pairings}
to compare the number of these matrices with
the number $N(\svec,\tvec)$ of $\{ 0,1\}$-matrices with row sums
$\svec$ and column sums $\tvec$.  
We do this by adapting
the results from~\cite{GMW} used to prove Theorem~\ref{main01}.
These calculations are carried out
in the pairing model, which is described in
Section~\ref{s:pairings}.   Our theorem follows
on combining Lemmas~\ref{easycases2},~\ref{2.3.2} and
Corollary~\ref{dsum}.
\end{proof}

In the semiregular case where $s_i = s$ for $1\leq i\leq m$
and $t_j = t$ for $1\leq j\leq n$, Theorem~\ref{main} says
the following.

\begin{corollary}
Suppose that $m,n\to \infty$ and that $sm=tn=S$ for
nonnegative integer functions $s=s(m,n)$, $t=t(m,n)$
and $S=S(m,n)$.
If\/ $1\le st=o(S^{2/3})$ then
\begin{align*}
M(m,s; n,t) &=\\
&\kern-15mm\frac{S!}{(s!)^m\, (t!)^n}
\exp\biggl(\frac{(s-1)(t-1)}{2} 
 - \frac{(s-1)(t-1)(2st - s - t - 10)}{12S}
 + O\biggl(\frac{s^3t^3}{S^2}\biggr)\biggr).
 \quad\qedsymbol
\end{align*}
\end{corollary}

For some applications the statement of Theorem~\ref{main}
is not very convenient.  In Section~\ref{s:alternative}
we will derive an alternative formulation,
very similar to one given for $N(\svec,\tvec)$ in~\cite{GMW}.
For $k=2,3$, define 
\begin{align*}
  \hat{\mu}_k &= \frac{mn}{S(mn+S)}\sum_{i=1}^m (s_i - S/m)^k\\[0.4ex]
  \hat{\nu}_k &= \frac{mn}{S(mn+S)}\sum_{j=1}^n (t_j - S/n)^k.
\end{align*}
To motivate the definitions, recall that $S/m$ is the mean value of $s_i$
and $S/n$ is the mean value of $t_j$, so these are scaled central moments.
We will prove Corollary~\ref{munu2}, stated in Section~\ref{s:alternative},
which has the following special case.

\begin{corollary}\label{nearreg}
Under the conditions of Theorem~\ref{main},
if $(1+\hat{\mu}_2)(1+\hat{\nu}_2)=O(S^{1/3})$ then
  \[
M(\svec,\tvec) =
  \frac{\displaystyle\prod_{i=1}^m\binom{n{+}s_i{-}1}{s_i}
   \prod_{j=1}^n\binom{m{+}t_j{-}1}{t_j}}
          {\displaystyle\binom{mn{+}S{-}1}{S}}\,
  \exp\biggl(  \dfrac12
  (1-\hat{\mu}_2)(1-\hat{\nu}_2)
              +O\biggl(\frac{st}{S^{2/3}}\biggr)\biggr).
              \quad\qedsymbol
\]
\end{corollary}

Corollary~\ref{nearreg} has an instructive interpretation.
Following~\cite{CMinteger},
we write $M(\svec,\tvec)= M P_1P_2 E$, where
\begin{align*}
M &= \binom{mn{+}S{-}1}{S},\quad
P_1 = M^{-1} \prod_{i=1}^m \binom{n{+}s_i{-}1}{s_i},\quad
P_2 = M^{-1} \prod_{j=1}^n \binom{m{+}t_j{-}1}{t_j},\\
E &= \exp\biggl(\dfrac12
  (1-\hat{\mu}_2)(1-\hat{\nu}_2)
              +O\biggl(\frac{st}{S^{2/3}}\biggr)\biggr).
\end{align*}
Clearly, $M$ is the number of $m\times n$ nonnegative
matrices whose entries sum to $S$.  In the uniform
probability space on these $M$ matrices, $P_1$ is the 
probability of the event that the row sums are given by $\svec$ 
and $P_2$ is the probability of the event that the column sums 
are given by $\tvec$.
The final quantity $E$ is thus a correction to account for
the non-independence of these two events.

Finally, in
Section~\ref{s:restricted} we show how to generalise 
Theorem~\ref{main} to matrices whose entries are restricted
to any subset of the natural numbers that includes 0 and 1.

A note on our usage of the $O(\,)$ notation in the following is in order.
Given a fixed
function $f(S) = o(S^{2/3})$, and any quantity $\phi$ that depends on
any of our variables,  $O(\phi)$ denotes any quantity whose absolute
value is bounded above by $\abs{c\phi}$ for some constant $c$ that 
depends on~$f$
\textit{and nothing else}, provided that $1\le st\le f(S)$.

\medskip

\emph{Note added in proof, 2011:}\ This version of the paper
the same as the journal version~\cite{GM08}, except as follows:
\begin{itemize}
\item 
Theorem~\ref{t:FM}, a statement of a special case of a more
general result from~\cite{FM}, was previously incomplete.  The
first inequality in (\ref{swine}) need only hold if $v$ is a sink,
but this condition was absent in~\cite{GM08}.
\item A note has been added at the end of the proofs of Lemmas~\ref{tswitch}
and~\ref{dswitch}, clarifying why it is valid to apply~\cite[Lemma 4.6]{GMW}
and~\cite[Lemma 4.8]{GMW} with a possibly larger value of $N_2$,
$N_3$ than used in~\cite{GMW}.
\item The proof of Lemma~\ref{restricted} has been changed to fix a small gap.
The old proof did not guarantee that $n_1(Q) = S - o(S)$
when $Q\in\mathcal{M}^-\setminus \mathcal{M}^\ast$.
The definition of $\mathcal{M}^-$ has changed and a new switching
argument is given to correct this.
\item We added a reference to the journal version~\cite{GM08} of this
paper.
\end{itemize}
Note that none of the statements of our own results from~\cite{GM08}
have changed.

\section{Switchings on matrices}\label{s:matrices}

In this section we will show
that the condition $st=o(S^{2/3})$ implies that most matrices
have no entries greater than~3.  We also find bounds on the
number of entries equal to~2 or~3.
Our tool will be the method of switchings, which we will
analyse using results of
Fack and McKay~\cite{FM} from which we
will distill the following special case.

\begin{theorem}\label{t:FM}
Let $G=(V,E)$ be a finite simple acyclic directed graph, with
each $v\in V$ being associated with a finite set~$C(v)$,
these sets being disjoint.
Suppose that $\S$ is a multiset of ordered pairs 
such that for each $(Q,R)\in \S$ there is an edge
$vw\in E$ with $Q\in C(v)$ and $R\in C(w)$. 
Further suppose that
$a,b : V\to\Reals$ are positive functions such
that, for each $v\in V$,
\begin{align}\label{swine}
  \begin{split}
  \Card{\{(Q,R)\in\S \mid Q\in C(v)\}} &\ge a(v)\,\card{C(v)}\,\,\,
   \text{ if $v$ is not a sink,}
  \\[0.5ex]
  \Card{\{(Q,R)\in\S \mid R\in C(v)\}} &\le b(v)\,\card{C(v)}\,,
  \end{split}
\end{align}
where the left hand sides are multiset cardinalities.
Let\/ $\emptyset\ne Y\subseteq V$. Then there is a directed
path $v_1,v_2,\ldots,v_k$ in $G$, where $v_1\in Y$ and
$v_k$ is a sink, such that
\begin{equation}\label{swopt}
  \frac{\sum_{v\in Y}\card{C(v)}}{\sum_{v\in V}\card{C(v)}}
  \le
  \frac{\sum_{v_i\in Y }N(v_i)}
       {\sum_{1\le i\le k}N(v_i)},
\end{equation}
where $N(v_i)$ is defined by
\begin{align*}
       N(v_1) &= 1,\\
       N(v_i) &= \frac{a(v_1)\cdots a(v_{i-1})}
                      {b(v_2)\cdots b(v_i)}& (2\le i\le k).
\end{align*}
\end{theorem}
\begin{proof}
 This follows from Theorems~1 and~2 of~\cite{FM}.
\end{proof}

\medskip

For $D\ge 2$, a \textit{$D$-switching} is described by the sequence
\[ 
  \(Q ; (i_0,j_0),(i_1,j_1),\ldots, (i_D,j_D)\)
\]
where $Q$ is a matrix in $\Mst$
and $(i_0,j_0),\, (i_1,j_1),\ldots, (i_D,j_D)$ is a $(D{+}1)$-tuple
of positions such that
\begin{itemize}\itemsep=0pt
\item the rows $i_0,\ldots, i_D$ are all distinct and the columns
$j_0,\ldots, j_D$ are all distinct;
\item there is a $D$ in position $(i_0,j_0)$ of $Q$;
\item the entries in positions $(i_\ell,j_\ell)$ of $Q$ are 
  not equal to~0 or $D+1$, for $1\leq \ell\leq D$;
\item there is a 0 in position $(i_\ell,j_0)$ and position
$(i_0,j_\ell)$ of $Q$ for $1\leq \ell\leq D$.
\end{itemize}
This $D$-switching transforms $Q$ into a matrix 
$R\in\Mst$ by acting on the $(D{+}1)\times (D{+}1)$
submatrix consisting of rows $(i_0,\ldots, i_D)$
and columns $(j_0,\ldots, j_D)$ as follows:
\[  Q = 
      \begin{pmatrix}
           & \\[-1ex]
           & D & 0   & 0    & \cdots  & 0 & \\
           & 0 & q_1 & \\
           & 0 &     & q_2 \\
       & \vdots &     &      & \ddots \\
          &  0 &     &      &        & q_D \\[-1ex]
          & 
       \end{pmatrix}
    ~\longmapsto~
       \begin{pmatrix}
            & \\[-1ex]
           & 0 & 1       & \!1       & \cdots  & \!1 & \\
           & 1 & q_1{-}1 & \\
           & 1 &         & \!q_2{-}1 \\
      & \vdots &         &         & \ddots \\
          &  1 &         &         &        & \!q_D{-}1\\[-1ex]
          &
       \end{pmatrix}
     = R\,.
\]
Matrix entries not shown can have any values and are unchanged by
the switching operation.
Notice that the $D$-switching preserves all row and column sums and 
reduces the number of entries equal to $D$ by at least 1 and
at most $D+1$.  The number of entries greater than~$D$ is unchanged.

\smallskip
A \textit{reverse $D$-switching}, which undoes a
$D$-switching (and vice-versa),
is described by a sequence
$\(R;(i_0,j_0),\ldots, (i_D,j_D)\)$
where $R\in\Mst$ and $(i_0,j_0),\, (i_1,j_1),\ldots, (i_D,j_D)$ is a
$(D{+}1)$-tuple of positions such that
\begin{itemize}\itemsep=0pt
\item the rows $i_0,\ldots, i_D$ are all distinct and the columns
$j_0,\ldots, j_D$ are all distinct;
\item there is a zero in position $(i_0,j_0)$ of $R$;
\item the entries in positions $(i_\ell,j_\ell)$ of $R$ are 
  not equal to~$D$, for $1\leq \ell\leq D$;
\item there is a 1 in position $(i_\ell,j_0)$ and position
$(i_0,j_\ell)$ of $R$ for $1\leq \ell\leq D$.
\end{itemize}

\nicebreak
\begin{lemma}\label{ab}
Let $D\ge 2$ and let $Q\in\Mst$ have at least $K\ge 2st$
non-zero entries that are not greater than~$D$, and at least
$J$ entries equal to~$D$.
Then there are at least $J(K-2st)^D$ $D$-switchings
and at most $S_D T_D$ reverse $D$-switchings that apply to~$Q$.
\end{lemma}

\begin{proof}
First consider $D$-switchings.
We want a lower bound on the number of
$(D{+}1)$-tuples $(i_0,j_0),\ldots, (i_D,j_D)$
of indices where a $D$-switching may be performed.
There are at least~$J$ ways to choose the position $(i_0,j_0)$.
Then we can choose the remaining positions one at a time, avoiding
choices which violate the rules.  The choice of the last position
$(i_D,j_D)$ is the most restricted, so we bound that.
By assumption, there are at least $K$ nonzero entries in $Q$
that are not greater than~$D$.
Of these we must exclude the entry in position $(i_0,j_0)$
as well as entries in the same column as a nonzero entry in row
$i_0$ other than column~$j_0$ (at most $(s-D)t$ positions),
entries in the same row as a nonzero
entry in column~$j_0$ other than row $i_0$ (at most $(t-D)s$ positions), and entries
in row $i_\ell$ or column $j_\ell$ for $1\leq \ell\leq D-1$ (at most
$(D-1)(s+t-2)$ positions).
Overall, we can choose position $(i_D,j_D)$ in at least
\[ K - 1 - (s-D)t - (t-D)s - (D-1)(s+t-2) \geq K - 2st \]
ways, and as we noted this also applies to
each of the less restricted positions $(i_\ell,j_\ell)$, where
$1\leq \ell < D$.  Hence at most $J(K - 2st)^D$
$D$-switchings involve~$Q$.

Next consider reverse $D$-switchings.
An ordered sequence of $D$ entries in the same row which
equal 1 may be chosen in at most $S_D$ ways, and an ordered sequence
of $D$ entries in the same column which equal 1 may be chosen in at
most $T_D$ ways.  Some of these choices will not give a legal position
for a reverse $D$-switching, but $S_DT_D$ is certainly an
upper bound.
\end{proof}

\medskip
Our first application of switchings will be to show that only
a vanishing fraction of our matrices have any entries greater
than~3.
For $j\ge 0$ and $D\ge 2$, let $\M_D(j)$ be the set of all
matrices in $\Mst$ with exactly $j$ entries equal to $D$ and none
greater than~$D$. Define $\M_D({>}0)=\bigcup_{j>0}\M_D(j)$,
and note that $\M_{D+1}(0) = \M_D(0)\cup\M_D({>}0)$.

\begin{lemma}\label{no4}
Suppose that $1\leq st=o(S^{2/3})$.
Let $U_1=U_1(\svec,\tvec)$ be the set of all matrices in $\Mst$
which contain an entry greater than~3.  Then
\[ 
 \card{U_1}/M(\svec,\tvec) = O(s^3t^3/S^2).
 \vadjust{\vskip-1ex}
\]
\end{lemma}

\begin{proof}
The largest possible entry of a matrix in $\Mst$
is $\Deltait = \min\{ s,t\}$. We will apply
Theorem~\ref{t:FM} to successively bound the possibility that the
maximum entry is~$D$, for $D=\Deltait,\Deltait-1,\ldots,4$.

\medskip

Fix $D$ with $4\le D\le\Deltait$.  
Define a directed graph $G=(V,E)$ with vertex set 
$V = \{ v_0, v_1, v_2,\ldots\,\}$
and edge set  $E=\{v_jv_i \mid  j-D-1\leq i\leq j-1\}$.
Associate each $v_i$ with the set $C(v_i) = \M_D(i)$.
Define $\S$ to be the set of pairs $(Q,R)$ related by a
$D$-switching, where $Q\in v_j, R\in v_i$ for some $v_jv_i\in E$.
Define $Y=\{ v_1, v_2,\ldots\,\}\subseteq V$.
Note that $S_DT_D>0$ since $D\le\Deltait$.

We can now use Theorem~\ref{t:FM} to bound
\[
 \frac{\card{\M_D({>}0)}}{\card{\M_{D+1}(0)}} = 
 \frac{\sum_{v\in Y}\card{C(v)}}{\sum_{v\in V}\card{C(v)}},
\]
once we have found
positive functions $a,b:V\to\Reals$ satisfying~\eqref{swine}.
These are provided by Lemma~\ref{ab} with $J=j$ and
$K=S/D$, the latter being clear since there are no entries
greater than $D$ and the total of all the entries is~$S$.
We have $S/D > 2st$ since 
$D\le\Deltait\le (st)^{1/2}$.
Thus we can take $a(v_j) = j(S/D - 2st)^D$ and
$b(v_j) = S_DT_D$.

Theorem~\ref{t:FM} tells us that, unless $\M_D({>}0)=\emptyset$,
there is a directed path $v_{t_1},v_{t_2},\ldots,v_{t_q}$,
where $q>1$ and $t_1>t_2>\cdots>t_q=0$ (since $v_0$ is
the only sink) such that~\eqref{swopt} holds.

Hence, using the values of $N$ as given in Theorem~\ref{t:FM} 
we have
\begin{align*}
 \frac{\card{\M_D({>}0)}}{\card{\M_{D+1}(0)}} &\leq
    \frac{N(v_{t_{q-1}}) + \cdots + N(v_{t_1})}
            {N(v_{t_q}) + \cdots + N(v_{t_2})} \\     
   &\leq \max_{1\leq i\leq q}\, \frac{N(v_{t_{i-1}})}{N(v_{t_i})} \\
   &= \max_{1\leq i\leq q}\, \frac{b(\M_D(t_i))}{a(\M_D(t_{i-1}))} \\
   &\leq \frac{S_DT_D}{(S/D-2st)^D}.
\end{align*}
Let $\xi_D$ denote this upper bound: that is,
$\xi_D=S_DT_D/(S/D-2st)^D$ for $4\leq D\leq\Deltait$.
Note that $\xi_4 = O(s^3t^3/S^2)$.  For 
$4\le D < \Deltait$, we have $\xi_D > 0$ and
\begin{align*}
\frac{\xi_{D+1}}{\xi_D} &\leq st\,\biggl(\frac{(D+1)^{D+1}}{D^D}\biggr)
               \frac{(S-2stD)^D}{(S-2st(D+1))^{D+1}} \\
       &= O(1)\, \frac{Dst}{S-2st(D+1)} 
          \biggl(1 - \frac{2st}{S-2stD}\biggr)^{\!\!-D} \\[0.5ex]
       &= o(1)
\end{align*}
uniformly over~$D$, where the last step uses the observation
that 
$\Deltait\le (st)^{1/2} = o(S^{1/3})$.

\medskip

Since $U_1=\M_4({>}0)\cup\M_5({>}0)\cup\cdots\cup\M_\Deltait({>}0)$
and $\M_{D+1}(0)\subseteq\Mst$ for $4\leq D\leq \Deltait$, 
we have $\card{U_1}/M(\svec,\tvec) 
\leq \xi_4 + \xi_5 + \cdots + \xi_\Deltait = O(s^3t^3/S^2)$
as required.
\end{proof}

\medskip
We may therefore restrict our attention to matrices with no
entry greater than~3.
Next we find upper bounds on the numbers of entries equal
to~2 or~3 which hold with high probability.

Define  
\begin{align*}
N_2 &= \begin{cases} 
    \,22 & \text{ if $S_2T_2 < S^{7/4}$,}\\[0.5ex]
\, \lceil \log  S \rceil & \text{ if $S^{7/4} \leq S_2T_2 < \tfrac{1}{5600} \, S^2\log S$,}\\[0.5ex] 
 \, \lceil 5600 S_2T_2/S^2\rceil & 
      \text{ if $\tfrac{1}{5600} \, S^2\log S\leq S_2 T_2$;} 
  \end{cases}\\[0.5ex]
N_3 &= \max\( \lceil \log  S \rceil, \lceil 230000 S_3T_3/S^3\rceil \).
\end{align*}
(Here and throughout the paper we have not attempted to optimise constants.)

We will use the following lemma.

\begin{lemma}
\label{useful}
Let $k$ be a positive integer and let $q$ and $n$
be positive real numbers such that $n \ge kq$.
Then
\[ n(n-q)\cdots (n-(k-1)q)\geq (n/e)^k.\]
\end{lemma}

\begin{proof}
Dividing the left side by $n^k$ gives, for $n>kq$,
\begin{align*}
\prod_{i=0}^{k-1} (1-iq/n)
       &=   \exp\biggl(\, \sum_{i=0}^{k-1} \log(1-iq/n) \biggr)\\
       &\geq \exp\biggl(\, \int_{0}^k \log(1-xq/n) \, dx \biggr) \\
       &=   \exp\(-k - (n/q-k)\log(1-kq/n))\\
       &\geq \exp(-k).
\end{align*}
The second line holds because $\log(1-xq/n)$ is a 
decreasing function for $x\in[0,k]$.
The case $n=kq$ follows by continuity.
\end{proof}

\begin{lemma}\label{notmany23}
Let $1\leq st = o(S^{2/3})$. 
Then, with probability $1-O(s^3t^3/S^2)$, a random element
of $\Mst$ has no entry greater than 3, at most~$N_3$
entries equal to 3, and at most $N_2$ entries equal to~2.
\end{lemma}

\begin{proof}
In view of Lemma~\ref{no4},
we may restrict our attention to the set $\M_4(0)$ of all matrices
in $\Mst$ with maximum entry at most 3.
We will start by applying 3-switchings as in 
Lemma~\ref{no4}
but the analysis will be more delicate.

In applying Theorem~\ref{t:FM} we have $V=\{ v_0,v_1,\ldots\,\}$,
with $v_h$ associated with $\M_3(h)$, and $Y=\{v_h \mid h > N_3\}$.
For sufficiently large $S$, we have from Lemma~\ref{ab} that
we can take $a(v_h)=\tfrac1{28}hS^3$ and $b(v_h)=S_3T_3$. 
If $S_3T_3=0$ then entries equal to 3 are impossible, so
we assume that $S_3T_3>0$.
Define $\varphi = 28 S_3T_3/S^3$.

According to Theorem~\ref{t:FM}, there is a sequence
\[ h_1 > h_2 > \cdots > h_q = 0,\]
with $h_1 > N_3$ and 
$h_{i-1} -4 \le h_{i} < h_{i-1}$ for all~$i$,
such that
\begin{align*}
 \frac{\card{\M_3({>}N_3)}}{\card{\Mst}}
    &\leq
          \frac{\card{\M_3({>}N_3)}}{\card{\M_4(0)}}\\
    &\leq \frac{N(h_\ell)  + N(h_{\ell-1}) + \cdots + N(h_1)}
             {N(h_{q}) + N(h_{q-1}) + \cdots + N(h_1)},
\end{align*}
where $\ell$ is the largest index such that $h_\ell\ge N_3+1$ and 
$N(h_i) = h_1\cdots h_{i-1} \varphi^{-i+1}$ for all~$i$.

Define $u=\lfloor\tfrac14\log S\rfloor$.
Since $N_3\ge\lceil\log S\rceil$, we have $\ell+u\le q$.
Also, for $0\le i\le\ell-1$,
\[ \frac{N(h_{\ell-i})}{N(h_{\ell+u-i})}
  \leq \frac{N(h_\ell)}{N(h_{\ell+u})}. \]
Therefore,
\begin{align*}
 \frac{\card{\M_3({>}N_3)}}{\card{\Mst}}
    &\leq \frac{N(h_\ell)  + N(h_{\ell-1}) + \cdots + N(h_1)}
             {N(h_{\ell+u}) + N(h_{\ell+u-1}) + \cdots N(h_{u+1})} \\
    &\leq \frac{N(h_\ell)}{N(h_{\ell+u})} \\
    &= \frac{\varphi^u}{h_\ell h_{\ell+1}\cdots h_{\ell+u-1}} \\
    &\leq \frac{\varphi^u}{(N_3+1)(N_3-3)\cdots(N_3-4u+5)}. 
\end{align*}
Since $N_3+1 > 4u$ we can apply Lemma~\ref{useful} 
to obtain the bound
\[
 \frac{\card{\M_3({>}N_3)}}{\card{\Mst}} 
  \leq \left(\frac{\varphi \, e}{N_3 + 1}\right)^{\!\!u}.
\]
Now $N_3\geq 230000 S_3T_3/S^3\geq 8214\, \varphi$,
and $u\ge\tfrac14\log S-1$,
so this upper bound is at most
\[
 \Bigl(\frac{e}{8214}\Bigr)^{\!\tfrac14\log S-1} 
  = O(1)S^{\frac14\log(e/8214)}\\
   = O(S^{-2}).
\]
This shows that with probability
$O(s^3t^3/S^2)$ there are at most $N_3$ entries equal to 3, as required.

\medskip
To bound the number of entries equal to~2, we proceed in
the same manner using 2-switchings, working under the
assumption that there are at most~$N_3$ entries equal to~3
and none greater than~3.
In applying Lemma~\ref{ab},
we can take $K=\tfrac12(S-3N_3)$, so that
$(K-2st)^2 \ge \tfrac15 S^2$ for sufficiently large~$S$.
Define $\psi=5S_2T_2/S^2$.
Arguing as above we find a sequence 
\[ d_1 > d_2 > \cdots > d_r = 0,\]
with the following properties: (i) $d_1 > N_2$ and 
$d_{i-1} -3\le  d_{i} < d_{i-1}$ 
for all~$i$, and
(ii)
if $p$ is the greatest integer such that $d_p>N_2$ then,
for any $w$ with $0<w\le r-p$,
the probability that there are
more than $N_2$ entries equal to~2, subject to there being at
most $N_3$ equal to~3, is bounded above by
\begin{equation}\label{2bound}
  \frac{\psi^w}{d_p d_{p+1}\cdots d_{p+w-1}}\,.
\end{equation}

First suppose that $S_2T_2 < S^{7/4}$, so that $N_2=22$ and
$\psi < 5S^{-1/4}$.
Since $d_p\geq N_2+1 = 23$, it follows that $r-p\geq 8$.
Taking $w=8$ in \eqref{2bound} gives 
\[ 
 \frac{\card{Y}}{\card{\Mst}}
    \leq \frac{\psi^8}{d_p d_{p+1}\cdots d_{p+7}}
   = O(S^{-2}).
\]

Now suppose that $S_2T_2\geq S^{7/4}$.
Then $N_2\geq \lceil\log S\rceil$
so we can take $w = \lfloor \tfrac13\log S\rfloor$.
Arguing as above by applying Lemma~\ref{useful} to~\eqref{2bound}, 
we obtain the bound $O(S^{-2})$ again.
This completes the proof.
\end{proof}

\medskip

{}From now on we proceed in two cases, as in~\cite{GMW}.
Say that the pair $(S_2,T_2)$ is \emph{substantial} if the
following conditions hold:
\begin{itemize}\itemsep=0pt
 \item $1\le st = o(S^{2/3})$,
 \item $S_2 \ge s\log^2 S$ and $T_2 \ge t\log^2 S$,
 \item $S_2T_2 \ge (st)^{3/2} S$.
\end{itemize}

\begin{lemma}\label{easycases}
If\/ $1\le st = o(S^{2/3})$ and $(S_2,T_2)$ is insubstantial, 
then with probability $1-O(s^3t^3/S^2)$, a random element of
$\Mst$ has no entry greater than 3, at
most one entry equal to 3 and at most two entries equal to 2.
\end{lemma}
\begin{proof}
The absence of entries greater than~3 follows from Lemma~\ref{no4}.
We can also, by Lemma~\ref{notmany23}, assume that the number
of entries equal to~2 or~3 is $o(S)$.
Therefore, most of the matrix entries are~0 or~1.
Let $\N$ be the set of all matrices in $\Mst$ with no entries
greater than 3, at most $N_2$
entries equal to 2 and at most $N_3$ entries equal to 3. 

To bound the number of entries equal to~2 or~3 even more tightly,
as this lemma requires, we employ $D$-switchings ($D=2,3$) with
the additional restriction that $q_1=\cdots=q_D=1$.
This ensures
that these \textit{restricted} $D$-switchings reduce the number
of entries equal to~$D$
by exactly one and do not create any new entries equal to~2 or~3.

Let $N''(h)$ be the number of matrices in $\N$ with $h$ 
entries equal to 3.  
If $Q$ is such a matrix then the number of restricted
$3$-switchings applicable to $Q$ is $hS^3(1+o(1))$
and the number of reverse restricted $3$-switchings is
at most $S_3T_3$. 
(This follows using arguments similar to those in Lemma~\ref{ab},
since there are $S-o(S)$ entries equal to 1.)
Therefore, if the denominator is nonzero,
\begin{equation}\label{insubs1}
   \frac{N''(h)}{N''(h{-}1)}
        = O(1) \frac{S_3T_3}{hS^3}.
\end{equation}
We can now easily check that each of the three causes
of insubstantiality (namely, $S_2 < s\log^2 S$, $T_2 < t\log^2 S$,
and 
$S_2T_2 < (st)^{3/2} S$) imply that
\[
   \frac{S_3T_3}{S^3} = O(s^{3/2}t^{3/2}/S) = o(1).
\]
Hence~\eqref{insubs1} implies that
\[
   \frac{\sum_{h\ge 2}N''(h)}{N''(0)}
      = O(s^3t^3/S^2).
\]
In precisely the same way, using restricted 2-switchings,
we find that
\[
   \frac{\sum_{d\ge 3}N'(d)}{N'(0)}
      = O(s^3t^3/S^2),
\]
where $N'(d)$ is the number of matrices in $\N$ with
$d$ entries equal to 2 and at most one entry equal to 3.  
The lemma follows.
\end{proof}

\section{From pairings to matrices}\label{s:pairings}

The remainder of the paper will involve calculations in the
\emph{pairing model}, which we now describe.
(This model is standard for working with random bipartite
graphs of fixed degrees:  see for example~\cite{McK}.)
Consider a set of $S$ {\it points\/} arranged in {\it cells\/}
$x_1,x_2,\ldots,x_m$, where cell $x_i$ has size~$s_i$ for $1\leq i\leq m$,
and another set of $S$ points
arranged in cells $y_1,y_2,\ldots,y_n$ where cell $y_j$ has size~$t_j$
for $1\leq j\leq n$.
Take a partition $P$ (called a {\it pairing\/}) of the $2S$ points
into~$S$ {\it pairs\/} with each pair having the form $(x,y)$
where $x\in x_i$ and $y \in y_j$ for some~$i,j$.
The set of all such pairings, of which there are $S!$,
will be denoted by~$\Pst$.
We work in the uniform probability space on~$\Pst$.

Two pairs are called \textit{parallel\/} if they involve the same
two cells.
A \textit{parallel class\/} is a maximal set of mutually
parallel pairs.
The \textit{multiplicity\/} of a parallel class (and of the
pairs in the class) is the cardinality of the class.
As important special cases,
a \textit{simple pair\/} is a parallel class of multiplicity one,
a \textit{double pair\/} is a parallel class of multiplicity two,
while a \textit{triple pair\/} is a parallel class of multiplicity
three.

Each pairing $P\in\Pst$ gives rise to a matrix
in $\Mst$ by letting the $(i,j)$-th entry of the matrix equal
the multiplicity of the parallel class from $x_i$ to $y_j$ in $P$. 

In~\cite{GMW} we noted that the number of pairings which gives
rise to each 0-1 matrix in $\Mst$ 
depends only on $\svec$ and $\tvec$ and is independent of the
structure of the matrix.  Hence the task of counting such
matrices reduces to finding the fraction of pairings that
have no multiplicities greater than~1.

More generally, matrices in $\Mst$ correspond to different
numbers of pairings.
For a pairing $P\in\Pst$, define the \textit{multiplicity vector\/}
of $P$ to be $\avec(P)=(a_2,a_3,\ldots\,)$
where~$a_r$ is the number of parallel classes of multiplicity~$r$.
Also define the \textit{weight} of $P$ as 
\[ w(P) = (2!)^{a_2} \,(3!)^{a_3} \,(4!)^{a_4}\cdots\, \]
For $Q\in\Mst$, define $w(Q)$ and $\avec(Q)$ to be the common
weight and multiplicity vectors of the pairings that yield~$Q$.

By elementary counting, a matrix $Q\in\Mst$ corresponds
to exactly
\[ \frac{1}{w(Q)}\, \prod_{i=1}^m s_i!\, 
 \prod_{j=1}^n \,t_j!\]
pairings in $\Pst$.
Therefore, if $A$ is a set of multiplicity vectors,
$\P_A=\{P\in\Pst\mid \avec(P)\in A\}$, and
$\M_A=\{Q\in\Mst\mid \avec(Q)\in A\}$, then
\begin{equation}\label{meM}
  \card{\M_A} = \frac{\sum_{P\in\P_A} w(P)}
   {\prod_{i=1}^m s_i!\,\prod_{j=1}^n \,t_j!}\,.
\end{equation}
This holds in particular if $A$ is the set of all nonnegative 
integer sequences, in which case $\P_A=\Pst$ and $\M_A=\Mst$.

\bigskip

We first prove Theorem~\ref{main} in the case that $(S_2,T_2)$ is
insubstantial.

\begin{lemma}\label{easycases2}
If\/ $1\le st = o(S^{2/3})$ and $(S_2,T_2)$ is insubstantial
then Theorem~\ref{main} holds.  
\end{lemma}

\begin{proof}
Similarly to~\cite[Lemma 2.2]{GMW}, define a \textit{doublet\/}
to be to be an unordered set of 2 parallel pairs.
A double pair provides one doublet, while a triple pair
provides~3 doublets.
For the uniform probability space over $\Pst$, let $b_r$ be
the expectation of the number of sets of $r$ doublets,
for $r\ge 0$.
In~\cite[Lemma 2.2]{GMW} it is shown that
\begin{align*}
b_0 &= 1,\\
b_1 &= \frac{S_2T_2}{2[S]_2},\\
b_2 &= \frac{S_3T_3}{2[S]_3} + 
   \frac{(S_2^2 - 4S_3 - 2S_2)(T_2^2 - 4T_3 - 2T_2)}{8[S]_4},\\
b_3 &= \frac{S_3T_3}{6[S]_3} + O(s^3t^3/S^2),\\
b_4 &= O(s^3t^3/S^2).
\end{align*}

Let $p_k$ denote the probability that a randomly chosen pairing
contains exactly $k$ doublets, for $k\geq 0$.  
Then
\[ p_k = \sum_{r\geq k} (-1)^{r+k}\, \binom{r}{k}\, b_r\]
and the partial sums of 
this series alternate above and below~$p_k$
(see for example~\cite[Theorem 1.10]{bollobas}).
Applying this, we find that
\begin{align*}
p_0 &= 1 - \frac{S_2T_2}{2[S]_2} + \frac{S_3T_3}{3[S]_3}
  + \frac{({S_2}^2 - 4S_3 - 2S_2)({T_2}^2 - 4T_3 - 2T_2)}{8[S]_4}
   + O(s^3t^3/S^2),\\
p_1 &=  \frac{S_2T_2}{2[S]_2} - \frac{S_3T_3}{2[S]_3} 
  - \frac{(S_2^2 - 4S_3 - 2S_2)(T_2^2 - 4T_3 - 2T_2)}{4[S]_4} 
   + O(s^3t^3/S^2),\\
p_2 &=  \frac{(S_2^2 - 4S_3 - 2S_2)(T_2^2 - 4T_3 - 2T_2)}{8[S]_4} 
                   + O(s^3t^3/S^2),\\
p_3 &=  \frac{S_3T_3}{6[S]_3} + O(s^3t^3/S^2).
\end{align*}
(The expression for $p_0$ was also derived in~\cite[Lemma 2.2]{GMW}.)
The configurations defining these cases are, respectively,
no parallel pairs, one double pair, two double pairs, and one
triple pair. 

Applying Lemma~\ref{easycases} and~\eqref{meM},
\begin{align*} M(\svec,\tvec) 
  &= \( 1+O(s^3t^3/S^2) \)
  \frac{S!}{\prod_{i=1}^m s_i!\,\prod_{j=1}^n \,t_j!}
  \( p_0 + 2p_1 + 4p_2 + 6p_3 \) \\
  &= 
  \frac{S!}{\prod_{i=1}^m s_i!\,\prod_{j=1}^n \,t_j!}
  \( p_0 + 2p_1 + 4p_2 + 6p_3 + O(s^3t^3/S^2) \) \\
  &= \frac{S!}{\prod_{i=1}^m s_i!\,\prod_{j=1}^n \,t_j!}\\
  &{\qquad}\times\biggl(
   1 + \frac{S_2T_2}{2[S]_2} + \frac{S_3T_3}{3[S]_3} 
  + \frac{(S_2^2 - 4S_3 - 2S_2)(T_2^2 - 4T_3 - 2T_2)}{8[S]_4} +
  O(s^3t^3/S^2)\biggr),
\end{align*}
where we have used the fact that $p_0+2p_1+4p_2+6p_3=1+o(1)$ in
the insubstantial case to get the second line.

This expression is equal to the expression in Theorem~\ref{main}
under our present assumptions.  (Note that since $(S_2,T_2)$ is 
insubstantial, the term $S_2^2T_2^2/2S^5$ which appears in the
statement of Theorem~\ref{main} is absorbed into the error term.)
\end{proof}

For nonnegative integers $d,h$, define $\C_{d,h}=
\C_{d,h}(\svec,\tvec)$ to be the set of all pairings in $\Pst$ with
exactly $d$ double pairs and $h$ triple pairs, but no parallel
classes of multiplicity greater than~3.
Also define
\[ w(\C_{d,h}) = \sum_{P\in\C_{d,h}} w(P)
     = 2^d \, 6^h\, \ac{d,h}. \]

A special case of~\eqref{meM}, used in~\cite{GMW}, is that the
number of 0-1 matrices in $\Pst$ is
\[
    N(\svec,\tvec) 
      = \frac{\ac{0,0}}
             {\prod_{i=1}^m s_i!\,\prod_{j=1}^n \,t_j!}\,.
\]
We will proceed by writing $M(\svec,\tvec)$ in terms of
$N(\svec,\tvec)$, as follows.

\begin{lemma}\label{2.3.2}
If $(S_2,T_2)$ is substantial then
\[ M(\svec,\tvec) = 
     N(\svec,\tvec)\, \sum_{d=0}^{N_2}\sum_{h=0}^{N_3}
         \frac{w(\C_{d,h})}{w(\C_{0,0})} \(1 + O(s^3t^3/S^2)\).
         \vadjust{\vskip-1ex}
\]
\end{lemma}
\begin{proof}
By Lemma~\ref{notmany23} and~\eqref{meM},
\begin{align*}
  M(\svec,\tvec) &= 
  \frac{1}{\prod_{i=1}^m s_i!\prod_{j=1}^n t_j!}
      \,\sum_{d=0}^{N_2}\sum_{h=0}^{N_3} w(\C_{d,h})
         \(1 + O(s^3t^3/S^2)\) \\
 &=
  N(\svec,\tvec)
      \sum_{d=0}^{N_2}\sum_{h=0}^{N_3}
        \frac{w(\C_{d,h})}{w(\C_{0,0})} \(1 + O(s^3t^3/S^2)\).\qedhere
\end{align*}
\end{proof}

We will evaluate the sum in Lemma~\ref{2.3.2} 
using two summation lemmas proved in~\cite{GMW} and restated below.

\begin{lemma}[{\cite[Corollary 4.3]{GMW}}]\label{sumcor}
Let $0\leq A_1\leq A_2$ and $B_1\leq B_2$ be real numbers.
Suppose that there exist integers $N$, $K$ with
$N\geq 2$ and  $0\leq K\leq N$, and a real number
$c> 2e$ such that $0\leq Ac<N-K+1$ and $\abs{BN}<1$
for all $A\in[A_1,A_2]$ and $B\in[B_1,B_2]$. 
Further suppose that there are real numbers $\delta_i$, for $1\le i\le N$,
and\/ $\gamma_i\ge 0$,  for $0\le i\le K$, such that\/
$\sum_{j=1}^i \abs{\delta_j}\le \sum_{j=0}^K\gamma_j\ff ij<\tfrac15$
for\/ $1\le i\le N$.\\
Given $A(1),\ldots,A(N)\in[A_1,A_2]$ and
$B(1),\ldots,B(N)\in[B_1,B_2]$,
define $n_0,n_1,\ldots,n_N$ by $n_0=1$ and
$$\frac{n_i}{n_{i-1}} = \frac {A(i)}i \(1 - (i-1)B(i)\)\(1+\delta_i)$$
for\/ $1\le i\le N$, with the following interpretation: if $A(i)=0$
then $n_j=0$ for $i\leq j\leq N$. 
Then  $$\varSigma_1 \le \sum_{i=0}^N n_i\le \varSigma_2,$$
where
\begin{align*}
 \varSigma_1 &= \exp\Bigl( A_1 - \tfrac12 A_1^2B_2 
     - 4 \sum_{j=0}^K\gamma_j(3A_1)^j\Bigr) - \tfrac14 (2e/c)^N,\\
 \varSigma_2 &= \exp\Bigl( A_2 - \tfrac12 A_2^2B_1 
 + \tfrac12 A_2^3B_1^2
     + 4 \sum_{j=0}^K\gamma_j(3A_2)^j\Bigr) + \tfrac14 (2e/c)^N.
     \quad\qedsymbol
\end{align*}
\end{lemma}

\begin{lemma}[{\cite[Corollary 4.5]{GMW}}]\label{sumcor2}
Let $N\geq 2$ be an integer and, for $1\leq i\leq N$, let real
numbers $A(i)$, $B(i)$ be given such that $A(i)\geq 0$ and
$1-(i-1)B(i) \ge 0$.
Define 
$A_1 = \min_{i=1}^N A(i)$, $A_2 = \max_{i=1}^N A(i)$,
$C_1 = \min_{i=1}^N A(i)B(i)$ and $C_2=\max_{i=1}^N A(i)B(i)$.
Suppose that there exists a real number $\hat{c}$ with 
$0<\hat{c} < \tfrac{1}{3}$ such that 
$\max\{ A/N,\, \abs{C}\} \leq \hat{c}$ 
for all $A\in [A_1,A_2]$,  $C\in [C_1,C_2]$.
Define $n_0,\ldots ,n_N$ by $n_0=1$ and
\[ \frac{n_i}{n_{i-1}} = \frac{A(i)}{i}\(1-(i-1)B(i)\)\]
for $1\leq i\leq N$, with the following interpretation: if $A(i)=0$
or $1-(i-1)B(i)=0$, then $n_j=0$ for $i\leq j\leq N$.  Then
\[ \varSigma_1 \leq \sum_{i=0}^N n_i\leq \varSigma_2\]
where
\begin{align*}
 \varSigma_1 &= \exp\( A_1 - \tfrac{1}{2} A_1 C_2 \)
               - (2e\hat{c})^N,\\
 \varSigma_2 &= \exp\( A_2 - \tfrac{1}{2} A_2 C_1 +
              \tfrac12 A_2 C_1^2 \) + (2e\hat{c})^N.
     \quad\qedsymbol
\end{align*}
\end{lemma}

We obtain bounds on the ratios we require by applying
results from~\cite{GMW}.  To begin with we focus
on the effect of changing the number of triple pairs while
keeping the number of double pairs fixed.

\begin{lemma}
\label{tswitch}
Suppose $0\le d\le N_2$ and $1<h\le N_3$, with $\C_{d,h}\neq \emptyset$.
If $(S_2,T_2)$ is substantial then
\[ \frac{w(\C_{d,h})}{w(\C_{d,h-1})}
   = \frac{S_3T_3 + O(s^2t^2(st+d+h)S)}{hS^3}.\]
\end{lemma}

\begin{proof}
This follows from~\cite[Lemma 4.6]{GMW} since, for $h\geq 1$,
\[ \frac{w(\C_{d,h})}{w(\C_{d,h-1})} =\frac{6\,\ac{d,h}}{\ac{d,h-1}}.
\]
Note that the values of $N_2$, $N_3$ used in this paper are no smaller than,
and are at most a constant
factor larger than, the values used in~\cite{GMW}.  For example,
we have $N_3 = \max\(\lceil \log S\rceil,\, \lceil 230000 S_3 T_3/S^3\rceil\)$,
while in~\cite{GMW} the value $\max\( \lceil \log S\rceil,\,
          \lceil 7 S_3 T_3/S^3\rceil\)$ was used.
Examination of the proof of~\cite[Lemma 4.6]{GMW} shows that the
bound given there for 
$\ac{d,h}/\ac{d,h-1}$
also holds for all $0\leq d\leq N_2$ and $1\leq h\leq N_3$. 
\end{proof}

\bigskip

Next, adapting the proof of~\cite[Corollary 4.7]{GMW} gives:

\begin{corollary}
\label{hsum}
\ Suppose $0\le d\le N_2$ with $\C_{d,0}\neq \emptyset$.
Further suppose that
$(S_2,T_2)$ is substantial.  Then
\[
\sum_{h=0}^{N_3} \frac{w(\C_{d,h})}{w(\C_{d,0})}
   = \exp\biggl( \frac{S_3T_3}{S^3}
            + O\(s^2t^2(st+d)/S^2\)\biggr).
\]
\end{corollary}

\begin{proof}
We will apply Lemma~\ref{sumcor2}.
Let $h'$ be the first value of $h\leq N_3$ for which $\C_{d,h}=\emptyset$, or
$h'=N_3+1$ if there is no such value.
Define $\alpha_h$, $1\le h < h'$, by
\begin{equation}\label{hrat}
\frac{\ac{d,h}}{\ac{d,h-1}} 
  =  \frac{S_3T_3 -  \alpha_h\(s^2t^2(st+d+(h-1)S)\)}{hS^3}.
\end{equation}
Lemma~\ref{tswitch} says that $\alpha_h$ is bounded independently
of $h$, $d$ and~$S$.

For $1\le h < h'$, define
\[ A(h)= \frac{S_3T_3 - \alpha_h(s^2t^2(st+d)S)}{S^3},\quad
   C(h) = \frac{\alpha_h s^2t^2}{S^2}.\]
If $\alpha_h \le 0$ then by definition
$A(h) \geq S_3T_3/S^3$, and $S_3T_3> 0$ 
since $h<h'$. Therefore $A(h) > 0$ in this case.
If $\alpha_h > 0$ then $C(h) > 0$, which implies that $A(h) > 0$ since the
right side of (\ref{hrat}) 
has the same sign as $A(h) - (h-1)C(h)$. 
Therefore $A(h)>0$ whenever $h<h'$.
Define $B(h)=C(h)/A(h)$ for $1\leq h < h'$.
Also define $A(h)=B(h)=0$ for $h'\le h\le N_3$.

Define
$A_1,A_2, C_1,C_2$ by taking the minimum and maximum of the
$A(h)$ and $C(h)$ over $1\leq h\leq N_3$, as in Lemma~\ref{sumcor2}.  
Let $A\in [A_1,A_2]$ and $C\in [C_1,C_2]$, and set
$\hat{c}=\tfrac1{41}$.
Since $A=S_3T_3/S^3+o(1)$ and $C=o(1)$, we have that
$\max\{A/N_3,\abs{C}\}<\hat{c}$  
for $S$ sufficiently large, by the definition of $N_3$.

Therefore Lemma~\ref{sumcor2} applies and says that
\[ \sum_{h=0}^{N_3} \frac{\ac{d,h}}{\ac{d,0}}
   = \exp\biggl( \frac{S_3T_3}{S^3}
            + O\(s^2t^2(st+d)/S^2\)\biggr) + 
                             O\((2e/41)^{N_3}\).\]
Finally, $(2e/41)^{N_3} \leq (2e/41)^{\log S} \leq S^{-2}$.
Since the sum we are estimating is at least equal to one, this additive
error term is covered by the error terms inside the exponential.
This completes the proof.
\end{proof}

\bigskip

Now we must sum over pairings with no triple pairs.

\begin{lemma}
\label{dswitch}
\ Suppose that $(S_2,T_2)$ is substantial and that $1\leq d\le N_2$
with $\C_{d,0}\neq \emptyset$.  Then
\[ 
 \frac{w(\C_{d,0})}{w(\C_{d-1,0})} 
  =  \frac{2A(d)}{d}\(1 - (d-1)B\)
        (1+ \delta_d)
\]
where
\begin{align*}
A(d) &= \frac{S_2T_2}{2S^2} \biggl(1 + \frac{S_2}{S^2} + \frac{T_2}{S^2} + 
 \frac{1}{S} + \frac{2S_3T_2}{S_2S^2} + \frac{2S_2T_3}{S^2 T_2}
- \frac{S_3T_3}{SS_2T_2} - \frac{2 S_2T_2}{S^3} \biggr)
  + O\biggl(\frac{s^3t^3}{S^2}\biggr),\\
B &= 
 \frac{2}{S_2} + \frac{2}{T_2} + \frac{4T_3}{T_2^2} +
       \frac{4S_3}{S_2^2} - \frac{4}{S},\\
\delta_d &= O\biggl(\frac{(d-1)^2s^2}{S_2^2} 
 + \frac{(d-1)^2t^2}{T_2^2} +
                \frac{dst(d+st)}{S_2T_2}\biggr).
\end{align*}
\end{lemma}

\begin{proof}
This follows from~\cite[Lemma 4.8]{GMW} since, for $d\geq 1$,
\[ \frac{w(\C_{d,0})}{w(\C_{d-1,0})} =\frac{2\,\ac{d,0}}{\ac{d-1,0}}. 
\]
As in Lemma~\ref{tswitch}, our value of $N_2$ is no smaller than, and is
at most a constant factor larger than, the value used
in~\cite{GMW}.  Examination of the proof of~\cite[Lemma 4.8]{GMW} shows
that the expression given there for $\ac{d,0}/\ac{d-1,0}$ also holds for
$1\leq d\leq N_2$.
\end{proof}

\bigskip
Adapting the proof of~\cite[Corollary 4.9]{GMW} gives the following:

\begin{corollary}
\label{dsum}
If $(S_2,T_2)$ is substantial then
\[ \sum_{d=0}^{N_2}\sum_{h=0}^{N_3} \frac{w(\C_{d,h})}{w(\C_{0,0})}
=
\exp\biggl( \frac{S_2T_2}{S^2} + \frac{S_2T_2}{S^3} 
                + O\biggl(\frac{s^3t^3}{S^2}\biggr) \biggr).
\]
\end{corollary}

\begin{proof}
We need to apply Lemma~\ref{sumcor} to
the result of Lemma~\ref{dswitch}, and take into account the terms coming
from the triple pairs (as given by Corollary~\ref{hsum}).

Let $d'$ be the first value of $d\leq N_2$ for which $\ac{d,0}=0$,
or $d'=N_2+1$ if no such value of $d$ exists.
Define $m_0,m_1,\ldots,m_{N_2}$ by
$$ m_d = 
  \frac{w(\C_{d,0})}{w(\C_{0,0})}\, \sum_{h=0}^{N_3} \frac{w(\C_{d,h})}{w(\C_{d,0})}
$$
for $0\leq d < d'$, and $m_d=0$ for $d'\leq d\leq N_2$.
Then clearly
\[ \sum_{d=0}^{N_2}\sum_{h=0}^{N_3} \frac{w(\C_{d,h})}{w(\C_{0,0})} =
\sum_{d=0}^{N_2} m_d.\] 
Corollary~\ref{hsum} tells us that for $d< d'$ we have
\begin{equation}
 m_d =  \frac{w(\C_{d,0})}{w(\C_{0,0})}\, 
  \exp\biggl( \frac{S_3T_3}{S^3} + O(s^3t^3/S^2) + \xi_d s^2t^2/S^2\biggl) 
\label{george}
\end{equation}
where $\xi_0=0$ and in general $\xi_d = O(d)$.   (Note that (\ref{george})
is also true for $d'\leq d\leq N_2$, since both sides
equal zero.)
If $\alpha$ is a constant
such that $\abs{\xi_d} \le \alpha d$ for $0\le d\le d'$, then
\begin{equation}\label{msum}
  \exp\biggl( \frac{S_3T_3}{S^3} + O(s^3t^3/S^2)\biggl)
  \sum_{d=0}^{N_2} n_d(-1) \le \sum_{d=0}^{N_2} m_d \le 
 \exp\biggl( \frac{S_3T_3}{S^3} + O(s^3t^3/S^2)\biggl)
 \sum_{d=0}^{N_2}n_d(1)
\end{equation}
where
$$ n_d(x) =  \frac{w(\C_{d,0})}{w(\C_{0,0})}\, \exp\(x \alpha d s^2t^2/S^2\).$$

Next we note that, for $x\in\{-1,1\}$, $n_0(x)=1$, and for
$1\le d\le d'$,
\[
  \frac{n_d(x)}{n_{d-1}(x)} = 2A(d)\(1 - (d-1)B\) \(1 + \delta_d\)
\]
with $A(d)$, $B$, and $\delta_d$ satisfying the expressions given in the
statement of Lemma~\ref{dswitch}. This follows since the factor
$\exp(x \alpha s^2t^2/S^2)$ is covered by the error term on~$A(d)$.
For $d'\leq d\leq N_2$ define $A(d)=0$.

\smallskip

Now let $A_1=A_1(x) = \min_d 2 A(d)$, $A_2=A_2(x)=\max_d 2 A(d)$,
where the maximum and minimum are taken over $1\leq d\leq N_2$.
Also let $B_1=B_2=B$, and $K=3$, and define $c=S^{1/4}$
if $S_2T_2<S^{7/4}$ and $c=41$ otherwise.
The conditions of Lemma~\ref{sumcor} now hold as we will show.
Let $A\in [A_1,A_2]$ be arbitrary.

Clearly $c>2e$.
If $S_2T_2 < S^{7/4}$ then $N_2=22$.  Using the condition
$S_2T_2\ge (st)^{3/2}S$ implied by the substantiality of
$(S_2,T_2)$, we find that $Ac = 1+o(1)$.
For $S_2T_2 \ge S^{7/4}$, $Ac = 41 S_2T_2/S^2(1+o(1))$.
It is also easy to check that $BN_2 = o(1)$.
Thus, in all cases we have that $Ac<N_2-2$ and
$\abs{BN_2}<1$ for sufficiently large~$S$.

If $d= O(S_2T_2/S^2)$ then
\[ \sum_{d=1}^{N_2} \abs{\delta_d} = 
            O\biggl(\frac{s^2 S_2T_2^3}{S^6} +
                    \frac{t^2 S_2^3T_2}{S^6} + 
                     \frac{stS_2^2T_2^2}{S^6} +
                                \frac{s^2t^2S_2T_2}{S^4}\biggr)
     = O\biggl(\frac{s^3t^3}{S^2}\biggr)
     = o(1),
\]
while if $d\leq \lceil \log S\rceil$ then
\[ \sum_{d=1}^{N_2} \abs{\delta_d}
           = O\biggl(\frac{s^2 \log^3 S}{S_2^2} + 
                            \frac{t^2\log^3 S }{T_2^2} +
                            \frac{st\log^3 S}{S_2T_2} + 
                         \frac{s^2t^2\log^2 S}{S_2T_2}\biggr)
            = o(1).
\]

Finally, for $1\le k\le N_2$, we have
\begin{align*}
 \sum_{d=1}^k \abs{\delta_d} &= 
    O\biggl(\, \sum_{d=1}^k (d-1)^2 
        \Bigl({\frac{s^2}{S_2^2}} + {\frac{t^2}{T_2^2}}\Bigr)\biggr)
    + O\biggl(\,\sum_{d=1}^k {\frac{d^2st}{S_2T_2}}\biggr) +  
       O\biggl(\,\sum_{d=1}^k {\frac{ds^2t^2}{S_2T_2}}\biggr) \\
  &= O\biggl(k(k-1)(2k-1)
        \Bigl({\frac{s^2}{S_2^2}} + {\frac{t^2}{T_2^2}}\Bigr)
       + {\frac{k(k+1)(2k+1)st}{S_2T_2}} + 
                         {\frac{k(k+1)s^2t^2}{S_2T_2}}\biggr)\\
    &\le \sum_{j=0}^K \gamma_j \ff kj,
\end{align*}
where
\[ \gamma_0=0,\
 \gamma_1 = O\biggl({\frac{s^2t^2}{S_2T_2}}\biggr),\
\gamma_2 = O\biggl({\frac{s^2}{S_2^2}} + {\frac{t^2}{T_2^2}} +
                     {\frac{s^2t^2}{S_2T_2}}\biggr),\
\gamma_3 = O\biggl({\frac{s^2}{S_2^2}} + {\frac{t^2}{T_2^2}} +
                     {\frac{st}{S_2T_2}}\biggr).\]
Since  $N_2^3(s^2/S_2^2 + t^2/T_2^2 + st/S_2T_2) = o(1)$,
which is easily checked, it 
follows that $\sum_{j=0}^K \gamma_j \ff kj< 1/5$ for $1\leq k\leq N_2$,
when $S$ is large enough.

\smallskip

Therefore the conditions of Lemma~\ref{sumcor} hold,
and we conclude that each of the bounds given by that
lemma for $\sum_{d=0}^{N_2} n_d(x)$ has the form
$$ \exp\biggl(A - \tfrac{1}{2}\,A^2B +
         O\Bigl(A^3B^2 + \sum_{j=0}^3 \gamma_j (3A)^j\Bigr)\biggr) +
              O\((2e/c)^{N_2}\), $$
where $A$ is either $A_1$ or $A_2$.
A somewhat tedious check shows that
$$O(A^3B^2)+\sum_{j=0}^3 \gamma_j (3A)^j=O(s^3t^3/S^2).$$

Next consider the error term $O\((2e/c)^{N_2}\)$.
If $N_2=22$ then $(2e/c)^{N_2} = (2eS^{-1/4})^{22} = O(S^{-2})$,
while in the other cases we have
$(2e/c)^{N_2} = (2e/41)^{N_2}\leq (2e/41)^{\log S} = O(S^{-2})$. 
Since $n_0=1$,
this additive error term is covered by a relative error of the same form.
Therefore, each of the bounds on $\sum_{d=0}^{N_2} n_d(x)$
has the form
\begin{align*}
  \exp\biggl(A - \tfrac{1}{2}\, A^2B + 
      O\biggl(\frac{s^3t^3}{S^2}\biggr)\biggr)
  &= \exp\biggl({\frac{S_2T_2}{S^2}} + {\frac{S_2T_2}{S^3}}
      - {\frac{S_3T_3}{S^3}} 
            + O\Bigl(\frac{s^3t^3}{S^2}\Bigr)\biggr).
\end{align*}
Modulo the given error terms, the final expression does not depend
on~$x$, nor on whether we are taking a lower bound or upper bound
in Lemma~\ref{sumcor}. 
To complete the proof, just apply~(\ref{msum}).
\end{proof}

\bigskip
Corollary~\ref{dsum} and Lemma~\ref{2.3.2} together prove
Theorem~\ref{main} in the substantial case.
The insubstantial case was already proved in Lemma~\ref{easycases2}.

\section{Alternative formulation}\label{s:alternative}

We now derive an alternative formulation of Theorem~\ref{main}.
Recall the definition
of $\hat{\mu}_k$ and $\hat{\nu}_k$ given in the Introduction.

\begin{corollary}\label{munu2}
Under the conditions of Theorem~\ref{main},
\begin{align*}
M(\svec,\tvec) &=
  \frac{\displaystyle\prod_{i=1}^m\binom{n{+}s_i{-}1}{s_i}
   \prod_{j=1}^n\binom{m{+}t_j{-}1}{t_j}}
          {\displaystyle\binom{mn{+}S{-}1}{S}}\\[0.5ex]
  &{\quad}\times \exp\biggl(  
  (1-\hat{\mu}_2)(1-\hat{\nu}_2)\biggl(\frac{1}{2}+
  \frac{3-\hat{\mu}_2\hat{\nu}_2}{4S}\biggr)
\\[0.5ex]
  &{\kern 4.5em} - \frac{(1-\hat{\mu}_2)(3+\hat{\mu}_2-2\hat{\mu}_2\hat{\nu}_2)}{4n}
       - \frac{(1-\hat{\nu}_2)(3+\hat{\nu}_2-2\hat{\mu}_2\hat{\nu}_2)}{4m}
  \\[0.5ex]
  &{\kern 4.5em} + \frac{(1-3{\hat{\mu}_2}^2+2\hat{\mu}_3)
           (1-3{\hat{\nu}_2}^2+2\hat{\nu}_3)}{12S}
              +O\biggl(\frac{s^3t^3}{S^2}\biggr)\biggr).
\end{align*}
\end{corollary}

\begin{proof}
By Stirling's formula or otherwise,
$$\binom{N{+}x{-}1}{x} = \frac{N^x}{x!}\exp\biggl(
    \frac{\ff x2}{2N} -  \frac{\ff x3}{6N^2}  -  \frac{\ff x2}{4N^2} +
        O(x^4/N^3)\biggr)$$
as $N\to\infty$, provided that the error term is bounded. 
This gives us the approximations
\begin{align*}
  \prod_{i=1}^{m} \binom{n{+}s_i{-}1}{s_i} &= \frac{n^S}{\prod_i s_i!}
     \exp\biggl( \frac{S_2}{2n} - \frac{S_2}{4n^2} - \frac{S_3}{6n^2}
     + O\biggl(\frac{s^3t^3}{S^2}\biggr)\biggr)\\[0.4ex]
  \prod_{j=1}^{n} \binom{m{+}t_j{-}1}{t_j} &= \frac{m^S}{\prod_j t_j!}
     \exp\biggl( \frac{T_2}{2m} - \frac{T_2}{4m^2} - \frac{T_3}{6m^2}
     + O\biggl(\frac{s^3t^3}{S^2}\biggr)\biggr)\\[0.4ex]
  \binom{mn{+}S{-}1}{S} &= \frac{(mn)^S}{S!}
     \exp\biggl( \frac{S^2}{2mn} - \frac{S}{2mn} - \frac{S^3}{6m^2n^2}
     + O\biggl(\frac{s^3t^3}{S^2}\biggr)\biggr).
\end{align*}
Substitute these expressions into Theorem~\ref{main}
and replace $S_2, S_3, T_2, T_3$ by their equivalents in terms of
$\hat{\mu}_2,\hat{\mu}_3,\hat{\nu}_2,\hat{\nu}_3$.
The desired result is obtained.
\end{proof}

\medskip

As noted in the Introduction, Theorem~\ref{main} establishes the
conjecture recalled after Theorem~\ref{CMmain} in some cases.
Using Corollary~\ref{munu2}, the following is easily seen.
(Note that $\hat{\mu}_2 = \hat{\mu}_3 = \hat{\nu}_2 = \hat{\nu}_3 = 0$ 
in the semiregular case.)

\begin{corollary}\label{Delta}
 If\/ $s=s(m,n)$ and $t=t(m,n)$ satisfy $ms=nt$ and
 $st=o\((mn)^{1/5}\)$, then
 \[ \Deltait(m,s;n,t)
   = \frac{5(s+t)}{6st}\(1+o(1)\).
        \quad\qedsymbol \]
\end{corollary}

Most of the terms inside the exponential of Corollary~\ref{munu2}
are tiny unless at least one of
$\hat{\mu}_2$, $\hat{\nu}_2$ 
is quite large (that is, the graph is very far from
semiregular).  In particular we can now prove Corollary~\ref{nearreg}
which was stated in the Introduction.

\begin{proof}[Proof of Corollary~\ref{nearreg}]
It is only necessary to check that the additional terms in
Corollary~\ref{munu2} have the required size.
It helps to realise that $\hat{\mu}_2\le s$,
$\abs{\hat{\mu}_3}\le s\hat{\mu}_2$, $\hat{\nu}_2\le t$
and $\abs{\hat{\nu}_3}\le t\hat{\nu}_2$.
\end{proof}

\bigskip

A random nonnegative $m\times n$ matrix with entries summing to~$S$
is just a random composition of $S$ into $mn$ parts.  (A composition
is an ordered sum of nonnegative numbers.)  
In particular, for $1\leq i\leq m$ the row sum $s_i$ satisfies
\[ \Pr(s_i = k) = \binom{k+n-1}{k}\binom{S - k + (m-1)n - 1}{S-k}
              \bigg/ \binom{S+mn-1}{S}  \quad\quad  (0\leq k\leq S).\]
{}From this we can 
compute the following expected values.
\begin{align*}
 \E\,\hat\mu_2 &= \frac{n(m-1)}{mn+1},
  & \E\,\hat\nu_2 &= \frac{m(n-1)}{mn+1}, \\
 \E\,\hat\mu_3 &= \frac{n(m-1)(m-2)(mn+2S)}{m(mn+1)(mn+2)}, &
 \E\,\hat\nu_3 &= \frac{m(n-1)(n-2)(mn+2S)}{n(mn+1)(mn+2)}.
\end{align*}
The first two expectations suggest that the argument of the
exponential in Corollary~\ref{nearreg} is close to~0 with high
probability for such a random matrix.  We will prove this in a
future paper, and note that the result gives a model for the
row and column sums of random matrices.

\section{Restricted sets of allowed entries}\label{s:restricted}

Given a subset $\J$ of the nonnegative integers,
let $\M(\svec,\tvec,\J)$ denote the set of 
matrices in $\Mst$ with all entries in the
set $\J$.  
Let $M(\svec,\tvec,\J)=\card{\M(\svec,\tvec,\J)}$.
By generalising the techniques of the preceding sections, we
can find an asymptotic expression for 
$M(\svec,\tvec,\J)$ whenever $0,1\in \J$.   

\begin{lemma} 
\label{restricted}
Let $\J\subseteq \mathbb{N}$ with $0,1\in\J$.
Define $\chi_2=0$ if\/ $2\notin\J$, $\chi_2=1$ if\/ $2\in\J$, 
and similarly $\chi_3$.
Then
\begin{align*}
 M(\svec,\tvec,\J) &= N(\svec,\tvec)
   \exp\biggl(\chi_2\frac{S_2T_2}{S^2}
             +\chi_2\frac{S_2T_2}{S^3}
             +(\chi_3-\chi_2)\frac{S_3T_3}{S^3}
             + O\biggl(\frac{s^3t^3}{S^2}\biggr)\biggr) \\
 &= \frac{S!}{\prod_{i=1}^m s_i!\, \prod_{j=1}^n t_j!}
   \exp\biggl((\chi_2-\tfrac12)\frac{S_2T_2}{S^2}
             +(\chi_2-\tfrac12)\frac{S_2T_2}{S^3}
             +(\chi_3-\chi_2+\tfrac13)\frac{S_3T_3}{S^3}\\
&\hspace{36mm}
      - \frac{S_2T_2(S_2+T_2)}{4S^4} - \frac{ S_2^2T_3+S_3T_2^2}{2S^4}
      + \frac{ S_2^2T_2^2}{2S^5} + O\biggl(\frac{s^3t^3}{S^2}\biggr)\biggr).
\end{align*}
\end{lemma}

\begin{proof}
Our general approach will be similar to that we used for
Theorem~\ref{main}, but
the methods of Section~\ref{s:matrices} will need significant
modification.  The source of the problem is that a $D$-switching
may introduce an entry that is not in~$\J$.

For $Q\in\Mst$ and $i\ge 1$, let $n_i(Q)$ be the number of entries
of~$Q$ equal to~$i$.  Also let $n_{\ge5}(Q)=\sum_{i\ge5} n_i(Q)$.
Define $N_2$ and $N_3$ as in Section~\ref{s:matrices} when
$(S_2,T_2)$ is substantial, and $N_2=2$ and $N_3=1$ when
$(S_2,T_2)$ is insubstantial.
For $Q\in\M^+$, let
\[  E^+(Q) = \sum_{D > \lceil (st)^{1/4}\rceil} n_D(Q),\qquad
   E^-(Q) = \sum_{D=5}^{\lceil (st)^{1/4}\rceil} n_D(Q).
\]
Consider the following subsets of~$\Mst$:
\begin{align*}
  \M^+ &= \M(\svec,\tvec,\J\cup\{4,5,6,\ldots\,\}), \\
  \M &= \M(\svec,\tvec,\J), \\
  \M^- &= \bigl\{ Q\in\M^+ \bigm| n_2(Q),\, n_3(Q),\, n_4(Q)\leq S^{5/6},
       \\
     & \hspace*{3cm} 
   \,\, E^-(Q) \leq \lceil 2(stS)^{1/2}\rceil, \,\,
        E^+(Q) \leq \lceil 2(st)^{1/4}\, S^{1/2}\rceil\},\\
  \M^\ast &= \bigl\{ Q\in\M(\svec,\tvec,\J\cap\{0,1,2,3\}) \bigm| n_2(Q)\le
                          N_2,\,\, n_3(Q)\le N_3\bigr\}.
\end{align*}
Also define the cardinalities $M^+, M, M^-, M^\ast$, respectively.
By monotonicity, we have $M^\ast\le M\le M^+$ and
$M^\ast\le M^-\le M^+$.

We now employ switchings to establish that $M^+{-}M^-< M^-{-}M^\ast$
and $M^-{-}M^\ast =O(s^3t^3/S^2)M^\ast$, from which it follows that
$M=\(1+O(s^3t^3/S^2)\)M^\ast$. {\bf (OK?)}

We start with the claim that $M^+{-}M^-< M^-{-}M^\ast$.  
Let $Q\in\M^+ - \M^-$ such that
$E^+(Q) > \lceil 2(st)^{1/4} S^{1/2}\rceil$.
We will use the following switching, illustrated by this
operation performed on submatrices:
\begin{equation}
\label{Qswitching}
 \begin{pmatrix} D_1 & 0\\ 0 & D_2\end{pmatrix}
   \mapsto \begin{pmatrix} D_1 - 1 &  1\\
                         1 & D_2 - 1
   \end{pmatrix}
\end{equation}
where $D_1, D_2\geq (st)^{1/4}$.  The number of forward
switchings is bounded below by
\[ E^+(Q)^2 - O\(st E^+(Q)\) = E^+(Q)(1-o(1)),\]
and the number of reverse switchings is bounded above by
\[ \frac{2stS}{(st)^{1/2}} = 2\sqrt{st}\, S.\]
Hence the number of reverse switchings divided by the
number of forward switchings is bounded above by
\[ 
  \frac{2(1+o(1))\sqrt{st} S}{E^+(Q)^2} \leq
   \frac{1+o(1)}{2}  < \dfrac{2}{3},
\]
using the assumed lower bound on $E^+(Q)$.
After repeatedly applying this switching, we reach a matrix $Q$
which satisfies 
\begin{equation}
\label{aim}
E^+(Q) \leq \lceil 2(st)^{1/4} S^{1/2}\rceil.
\end{equation}

The next switching is applied to matrices $Q\in \M^+$ for
which (\ref{aim}) holds 
but $E^-(Q) > \lceil 2(stS)^{1/2}\rceil$.
The switching that we used is the same as that shown in
(\ref{Qswitching}) except that now 
$D_1,D_2\in \{ 5,\ldots, \lceil (st)^{1/4}\rceil\}$.
The number of forward switchings is bounded below by
\[ E^-(Q)^2 - O\( st E^-(Q)\) = E^-(Q)^2(1-o(1)),\]
and the number of reverse switchings is bounded above
by
\[ 2st S.\]
Hence the number of reverse switchings divided by the number of
forward switchings is bounded above by
\[ \frac{2(1+o(1))stS}{E^-(Q)^2} \leq \frac{1+o(1)}{2} < \dfrac{2}{3}.
\]
We apply this switching until we reach a matrix $Q$ which 
satisfies both (\ref{aim})  and
\begin{equation}
\label{aim2}
E^-(Q) \leq \lceil 2(stS)^{1/2}\rceil.
\end{equation}

To analyse these two switchings using Theorem~\ref{t:FM},
we can define the sets
\[ C(i) = \{ Q\in \M^+ - \M^- \mid \sum_{D\geq 2} D n_D(Q) = i\}.\]
If $Q\in C(i)$ and $R$ can be obtained from $Q$ using one of the
switchings described above, then $R\in C(i-2)$.  This leads
to an acyclic directed graph in each case, and we have shown above
that all the ratios $b(v_i)/a(v_{i-1})$ in Theorem~\ref{t:FM}
are at most $2/3$.

Next we need to reduce each of $n_2(Q), n_3(Q), n_4(Q)$ to
below $S^{5/6}$.  We achieve this using a succession of three types of
switchings, illustrated by the following operations on submatrices:
for example, the switching
\[
  \begin{pmatrix} 4 & 0 & 0 & 0\\         
                  0 & 4 & 0 & 0\\
                  0 & 0 & 4 & 0\\
                  0 & 0 & 0 & 4\end{pmatrix}
          \mapsto
                  \begin{pmatrix} 1 & 1 & 1 & 1\\ 1 & 1 & 1 & 1\\
                             1 & 1 & 1 & 1\\ 1 & 1 & 1 & 1\end{pmatrix}
\]
will be used to reduce $n_4(Q)$
(with analogous operations for $D=2,3$).
First we apply the
switching for $D=4$ until $n_{4}(Q)\le S^{5/6}$, then the switching
for $D=3$ until $n_3(Q)\le S^{5/6}$, finally applying
the switching for $D=2$ until $n_{2}(Q)\le S^{5/6}$.  
As a representative example, take the switching for $D=4$.
By counting similarly to Lemma~\ref{ab}, this switching can
be applied to $Q$ in at least $(n_4(Q)-O(st))^4$ ways, and the
inverse can be applied in at most $Ss^3t^3$ ways.  For
$n_4(Q)>S^{5/6}$, the condition $s^3t^3=o(S^2)$ implies that
$Ss^3t^3=o\((n_4(Q)-O(st))^4\)$, so the ratios 
denoted by $b(v_i)/a(v_{i-1})$ in Theorem~\ref{t:FM} are all 
$o(1)$.  

Since none of the 
switchings can undo the work of a previous switching, the end
result is a matrix $\M^-\setminus\M^\ast$. 
(Note that in the resulting matrix $R$,  at least one of
$E^+(R)$, $E^-(R)$, $n_2(R)$, $n_3(R)$ or $n_4(R)$ will be
just under the threshold value. 
This implies that $R\not\in\M^\ast$.)
This establishes the bound $M^+{-}M^-< M^-{-}M^\ast$.

For any matrix $Q\in\M^-\setminus\M^\ast$ we have
\[ \sum_{D\geq \lceil (st)^{1/4}\rceil} D n_D(Q)
  \leq \min\{s,t\} E^+(Q) \leq 3(st)^{3/4}\, S^{1/2},
\]
since using (\ref{aim}) and since $\min\{ s,t\} \leq (st)^{1/2}$.  
Similarly, (\ref{aim2}) implies that
\[ \sum_{D=5}^{\lceil (st)^{1/4}\rceil} D n_D(Q)
  \leq \lceil (st)^{1/4}\rceil E^-(Q) \leq  3(st)^{3/4} S^{1/2},
\]
which leads to
\[ \sum_{D\geq 2} Dn_D(Q) \leq 6(st)^{3/4} S^{1/2} + 3S^{5/6} = o(S).
\]
Hence when $Q\in\M^-\setminus\M^\ast$, we know that $n_1(Q)=S-o(S)$.
We can now continue precisely as in Lemmas~\ref{no4},~\ref{notmany23},
using $D$-switchings restricted to $q_1=\cdots=q_D=1$.
This restriction ensures that $D$-switchings only create entries
with value equal to 0 or 1.  The various switching counts can be
taken as essentially the same as before, since
all but a vanishing fraction of the non-zero entries are~1.
We conclude that $M^-{-}M^\ast =O(s^3t^3/S^2)M^\ast$ which, as noted above, 
implies that $M=\(1+O(s^3t^3/S^2)\)M^\ast$.

Having now reduced the task to evaluation of $M^\ast$, we can
complete the proof following Lemma~\ref{easycases2} in the
insubstantial case, and Section~4 in the substantial case.
In Lemma~\ref{easycases2} the only modification is to replace
the expression $p_0 + 2p_1 + 4p_2 + 6p_3$ by
$p_0 + 2\chi_2\, p_1 + 4\chi_2\, p_2 + 6\chi_3\, p_3$.

Now suppose that $(S_2,T_2)$ is substantial.
If $\chi_2=\chi_3$ then the result is given by either
Theorem~\ref{main01} or Theorem~\ref{main}.
If $\chi_2=0$ and $\chi_3=1$ then the result follows from applying
Corollary~\ref{hsum} with $d=0$, since arguing as in Lemma~\ref{2.3.2}
shows that
\[ M(\svec,\tvec,\J) = 
     N(\svec,\tvec)\, \sum_{h=0}^{N_3}
         \frac{w(\C_{0,h})}{w(\C_{0,0})} \(1 + O(s^3t^3/S^2)\)
         \vadjust{\vskip-1ex}
\]
in this case.  
Finally, if $\chi_2=1$ and $\chi_3=0$
then
\[ M(\svec,\tvec,\J) = 
     N(\svec,\tvec)\, \sum_{d=0}^{N_2}
         \frac{w(\C_{d,0})}{w(\C_{0,0})} \(1 + O(s^3t^3/S^2)\)
         \vadjust{\vskip-1ex}
\]
so in place of (\ref{george}) we simply have
$m_d =  w(\C_{d,0})/w(\C_{0,0}) = n_d(0)$ for $0\leq d\leq N_2$.
The remainder of the proof is identical except that there is
no need to apply (\ref{msum}) at the~end.
\end{proof}


\begin{thebibliography}{999}

 \bibitem{BSY}
 A.~Barvinok, A.~Samorodnitsky and A.~Yong,
 Counting magic squares in quasi-polynomial time, preprint (2007);
 \texttt{http://www.arxiv.org/abs/math/0703227}.

\bibitem{bekessy} A.~B\'ek\'essy, P.~B\'ek\'essy and J.~Koml\'os,
  Asymptotic enumeration of regular matrices,
  Studia Sci. Math. Hungar., 7 (1972) 343--353.

\bibitem{bender} E.\,A.~Bender,
  The asymptotic number of nonnegative integer matrices with given row
  and column sums, Discrete Math., 10 (1974) 345--353.

\bibitem{bollobas}
B.~Bollob\'as, Random Graphs (2nd edn.), 
Cambridge University Press, Cambridge, 2001.

\bibitem{CMinteger} E.\,R.~Canfield and B.\,D.~McKay,
Asymptotic enumeration of integer
matrices with large equal row and column sums,
Combinatorica, 30 (2010) 655--680.

\bibitem{DE}
P.~Diaconis and B.~Efron, 
Testing for independence in a two-way table: new interpretations
of the chi-square statistic (with discussion), 
Ann. Statist., 13 (1995) 845--913.

\bibitem{DGang} P.~Diaconis and A.~Gangolli,
  Rectangular arrays with fixed margins, 
in:
Discrete Probability and Algorithms,
IMA Volumes on Mathematics and its Applications, 
vol.~72, Springer, New York, (1995),
pp.~15--41.

\bibitem{DKM} M.~Dyer, R.~Kannan and J.~Mount,
Sampling contingency tables, Random Structures Algorithms,
10 (1997), 487--506.

\bibitem{everett} C.\,J.~Everett, Jr., and P.\,R.~Stein,
  The asymptotic number of integer stochastic matrices,
   Discrete Math., 1 (1971) 33--72.

\bibitem{FM}
V.~Fack and B.\,D.~McKay,
A generalized switching method for combinatorial estimation,
Australas. J. Combin., 39 (2007), 141--154.

\bibitem{GM08}
C.~Greenhill and B.\,D.~McKay,
Asymptotic enumeration of sparse nonnegative integer matrices
with specified row and column sums,
Advances in Applied Mathematics, 41 (2008), 459--481.

\bibitem{GMW} C.~Greenhill, B.\,D.~McKay and X.~Wang,
 Asymptotic enumeration of sparse \mbox{0-1} matrices with
irregular row and column sums,
 J. Combin. Theory Ser. A, 113 (2006) 291--324.

\bibitem{McK} B.\,D.~McKay,
Asymptotics for 0-1 matrices with prescribed line sums,
in: Enumeration and Design, Academic Press,
Toronto, 1984, pp.~225--238.

 \bibitem{morris}
 B.~Morris,
 Improved bounds for sampling contingency tables,
 Random Structures Algorithms, 21 (2002)
 135--146.

\bibitem{read} R.\,C.~Read,
  Some enumeration problems in graph theory,
  Doctoral Thesis, London University, (1958).

\bibitem{stanley} R.\,P.~Stanley,
Combinatorics and Commutative Algebra, 
Progress in Mathematics, vol.~41, 
Birkh\"auser, Boston, 1983.

\end{thebibliography}
\end{document}